\documentclass[10pt]{amsart}
\usepackage{amssymb}
\usepackage{amsthm,amsmath}
\usepackage{cite}

\usepackage{MnSymbol}

\title[Newton Series Representation]{Newton Series Representation of Completely Monotone Functions}

\author{Thomas Lamby}
\address{University of Luxembourg, Department of Mathematics, Maison du Nombre, 6, avenue de la Fonte, L-4364 Esch-sur-Alzette, Luxembourg}
\email{thomas.lamby[at]uni.lu}

\author{Jean-Luc Marichal}
\address{University of Luxembourg, Department of Mathematics, Maison du Nombre, 6, avenue de la Fonte, L-4364 Esch-sur-Alzette, Luxembourg}
\email{jean-luc.marichal[at]uni.lu}

\author{Na\"im Zena\"idi}\thanks{Corresponding author: Na\"im Zena\"idi, University of Li\`ege, Department of Mathematics, All\'ee de la D\'ecouverte, 12 - B37, B-4000 Li\`ege, Belgium. Email: nzenaidi[at]uliege.be}
\address{University of Li\`ege, Department of Mathematics, All\'ee de la D\'ecouverte, 12 - B37, B-4000 Li\`ege, Belgium}
\email{nzenaidi[at]uliege.be}

\date{October 13, 2025}

\begin{document}

\theoremstyle{plain}

\newtheorem{theorem}{Theorem}[section]
\newtheorem{lemma}[theorem]{Lemma}
\newtheorem{proposition}[theorem]{Proposition}
\newtheorem{corollary}[theorem]{Corollary}
\newtheorem{fact}[theorem]{Fact}

\theoremstyle{definition}

\newtheorem{definition}[theorem]{Definition}

\newtheorem{ex}[theorem]{Example}
    \newenvironment{example}
    {\renewcommand{\qedsymbol}{$\lozenge$}
    \pushQED{\qed}\begin{ex}}
    {\popQED\end{ex}}

\newtheorem{rem}[theorem]{Remark}
    \newenvironment{remark}
    {\renewcommand{\qedsymbol}{$\lozenge$}
    \pushQED{\qed}\begin{rem}}
    {\popQED\end{rem}}

\theoremstyle{remark}

\newtheorem{claim}{Claim}

\newcommand{\R}{\mathbb{R}}
\newcommand{\N}{\mathbb{N}}
\newcommand{\Z}{\mathbb{Z}}
\newcommand{\C}{\mathbb{C}}
\newcommand{\id}{\mathrm{id}}
\newcommand{\cC}{\mathcal{C}}
\newcommand{\cD}{\mathcal{D}}
\newcommand{\cK}{\mathcal{K}}
\newcommand{\cAM}{\mathrm{AM}}
\newcommand{\cCM}{\mathrm{CM}}
\newcommand{\cRM}{\mathrm{RM}}
\newcommand{\ran}{\mathrm{ran}}

\def\tchoose#1#2{{\textstyle{{{#1}\choose{#2}}}}}

\begin{abstract}
We prove that every completely monotone function defined on a right-unbounded open interval admits a Newton series expansion at every point of that interval. This result can be viewed as an analog of Bernstein's little theorem for absolutely monotone functions. As an application, we use it to study principal indefinite sums, which are constructed via a broad generalization of Bohr-Mollerup's theorem.
\end{abstract}

\subjclass[2020]{Primary: 26A48, 26A51, 39A70, 41A58; Secondary: 33B15, 39A12, 40A30.}

\keywords{Completely monotone function, Absolutely monotone function, Real analyticity, Newton series expansion, Principal indefinite sum, Higher-order convexity}

\maketitle
\section{Introduction}

A real-valued function $f\colon\, I\to\R$, defined on an open interval $I$, is called \emph{completely monotone} if it is infinitely differentiable and satisfies
$$
(-1)^n f^{(n)}(x) ~\geq ~0 \quad\text{for all $x\in I$ and all $n\in\N$},
$$
where $f^{(n)}$ denotes the $n$th derivative of $f$, and $\N$ stands for the set of nonnegative integers. As an example, the function $f\colon\,\R_+\to\R$, defined on the open half-line $\R_+=(0,\infty)$ by $f(x)=1/x$, is completely monotone.

This concept, introduced by Bernstein \cite{Ber14} in 1914 (see also \cite{Ber29}), has several noteworthy consequences. Among them, the following two results are particularly significant:

\smallskip

\begin{list}{$\bullet$}{\setlength{\leftmargin}{3.5ex} \setlength{\itemindent}{0ex}}
  \item\emph{Bernstein's little theorem}: Any completely monotone function $f\colon\, I\to\R$, defined on an open interval $I$, is real analytic. That is, for any point $a\in I$, there exists an open neighborhood $U\subset I$ of $a$ such that
    $$
    f(x) ~=~ \sum_{k=0}^{\infty}\frac{f^{(k)}(a)}{k!}\, (x-a)^k\qquad (x\in U),
    $$
    which means that $f$ coincides with its Taylor series expansion about $a$. We revisit this result in Section 2.
  \item\emph{Bernstein's theorem on monotone functions}: A function $f\colon\,\R_+\to\R$ is completely monotone if and only if it is representable as a Laplace type integral of the form
    $$
    f(x) ~=~ \int_0^{\infty}e^{-xt}\, d\mu(t)\qquad (x>0),
    $$
    where $\mu$ is the Lebesgue-Stieltjes measure induced by an increasing function from the interval $[0,\infty)$ into itself.
\end{list}

\smallskip

We note that many familiar functions from calculus are completely monotone, and the two theorems above then yield classical series and integral representations for these functions. For further background, see, for example, Boas~\cite{Boa71}, Krantz and Parks~\cite{KraPar02}, Mahajan and Ross \cite{MahRos82}, and Schilling \emph{et al.}~\cite{SchSonVon12}.

In this paper, we offer an alternative series representation for completely monotone functions. Specifically, in Section 3, we establish the remarkable fact that \emph{every completely monotone function $f\colon\, I\to\R$, defined on a right-unbounded open interval $I$, admits a Newton series expansion at every point $a\in I$ that is valid throughout $I$}. More precisely, for any point $a\in I$, the following identity holds:
$$
f(x) ~=~ \sum_{k=0}^{\infty}\frac{\Delta^k f(a)}{k!}\, (x-a)^{\underline{k}}\qquad (x\in I),
$$
or, equivalently,
\begin{equation}\label{eq:New576}
f(x) ~=~ \sum_{k=0}^{\infty}\tchoose{x-a}{k}\,\Delta^k f(a)\qquad (x\in I),
\end{equation}
where $\Delta$ denotes the classical forward difference operator, and where the \emph{falling factorial power} $x^{\underline{k}}$ (see Graham \emph{et al.}\ \cite[pp.~47--48]{GraKnuPat94}) is defined by
$$
x^{\underline{k}} ~=~ x(x-1)~\cdots ~(x-k+1) ~=~ k!\,\tchoose{x}{k}\qquad (x\in\R,~k\in\N). 
$$
We actually show that this result still holds under the weaker assumption that only a higher-order derivative of $f$ is completely monotone. It is also noteworthy that the series expansion in \eqref{eq:New576} is valid not merely in a neighborhood of $a$, but throughout the entire interval $I$. Furthermore, we observe that the requirement for $I$ to be right-unbounded arises naturally from the fact that the Newton series expansion at $a\in I$ evaluates the function at the points $a+k$ for every $k\in\N$.

Possessing a Newton series expansion is a fairly restrictive condition; many analytic functions---such as the classical exponential function $f(x)=e^x$---fail to admit one (see Remark~\ref{rem:BinThm}). However, in Section 4, we show that functions belonging to a notable class---namely, the \emph{principal indefinite sums}, which were recently introduced via a significant generalization of the Bohr-Mollerup theorem and thoroughly examined in \cite{MarZen22, MarZen24}---typically admit such expansions.

\section{Analyticity of Completely Monotone Functions}

Recall that a real-valued function $f\colon\, I\to\R$, defined on an open interval $I$, is called \emph{absolutely monotone} if it is infinitely differentiable and satisfies
$$
f^{(n)}(x) ~\geq ~0 \quad\text{for all $x\in I$ and all $n\in\N$.}
$$
Classical examples of absolutely monotone functions include $f(x)=e^x$, when defined on $\R$, and $g(x)=x^2$, when defined on $\R_+$.

This concept is closely related to that of \emph{completely monotone} functions. In fact, we can readily see that a function $f(x)$ is completely monotone on $I$ if and only if the function $f(-x)$ is absolutely monotone on the reflected interval
$$
{-I} ~=~ \{-x : x\in I\}.
$$

In this section, we present and prove Bernstein's little theorem under a slightly relaxed assumption: namely, that the function $f$, or one of its derivatives, is absolutely monotone. The proof presented here is a straightforward adaptation of the arguments given in Krantz and Parks \cite[Theorem 3.4.1]{KraPar02}. We include it in full both for its intrinsic interest and because similar techniques will be used in the next section to establish Newton series representations of completely monotone functions.

In light of the preceding observations, Bernstein's little theorem asserts that \emph{a sufficient condition for an infinitely differentiable function on an open real interval to be real analytic is that one of its derivatives is absolutely or completely monotone}. For example, the logarithm function $f(x)=\ln x$ has a completely monotone derivative on $\R_+$, and is thus real analytic on its domain of definition.

\begin{theorem}[Bernstein's Little Theorem]\label{thm:21BerLiTh}
Let $I$ be a real open interval. Suppose that a function $f\colon\, I\to\R$ is infinitely differentiable and that $f^{(q)}$ is absolutely monotone for some $q\in\N$. Then $f$ is real analytic on $I$.
\end{theorem}

\begin{proof}
For any $a,x\in I$ and any $n\in\N^*{=\, }\N\setminus\{0\}$, with $n\geq q$, the classical Taylor theorem (with the remainder in integral form) states that
\begin{equation}\label{eq:TayEq554}
f(x) - \sum_{k=0}^{q-1}\frac{f^{(k)}(a)}{k!}\, (x-a)^k ~=~ \sum_{k=q}^{n-1}\frac{f^{(k)}(a)}{k!}\, (x-a)^k + R_n(x),
\end{equation}
where
$$
R_n(x) ~=~ \int_a^x\frac{f^{(n)}(t)}{(n-1)!}\, (x-t)^{n-1}\, dt.
$$
Using the change of variable $t=a+s(x-a)$, the remainder term $R_n(x)$ can be equivalently expressed as
$$
R_n(x) ~=~ \frac{(x-a)^n}{(n-1)!}\,\int_0^1 f^{(n)}\big(a+s(x-a)\big)\, (1-s)^{n-1}\, ds.
$$

Now, fix $a\in I$ and let us prove that $f$ is real analytic at $a$. To this extent, we fix $\varepsilon >0$ and $x\in I$ such that $x\in (a-\varepsilon,a+\varepsilon)\subset I$. Let also $b\in I$ with $x<b$ and $|x-a|<b-a$. Since $f^{(n+1)}\geq 0$, it follows that $f^{(n)}$ is increasing on $I$. Therefore, we obtain the estimate
$$
|R_n(x)| ~\leq ~ \frac{|x-a|^n}{(n-1)!}\,\int_0^1 f^{(n)}\big(a+s(b-a)\big)\, (1-s)^{n-1}\, ds,
$$
or equivalently,
$$
|R_n(x)| ~\leq ~ \Big|\frac{x-a}{b-a}\Big|^n\, R_n(b).
$$
Since the identity in \eqref{eq:TayEq554} also holds when $x=b$, we obtain:
$$
f(b)-\sum_{k=0}^{q-1}\frac{f^{(k)}(a)}{k!}\, (b-a)^k-R_n(b) ~ \geq ~ 0,
$$
and hence
$$
0 ~\leq ~ |R_n(x)| ~\leq ~ \Big|\frac{x-a}{b-a}\Big|^n\, \Big(f(b)-\sum_{k=0}^{q-1}\frac{f^{(k)}(a)}{k!}\, (b-a)^k\Big).
$$
This shows that the sequence $n\mapsto R_n(x)$ converges pointwise to zero for any $x\in (a-\varepsilon,a+\varepsilon)$, and hence the function $f$ is real analytic at $a$. Since $a$ was chosen arbitrary in $I$, we conclude that $f$ is real analytic on the entire interval $I$.
\end{proof}

\begin{remark}\label{rem:entireRe}
We note that the proof of Theorem~\ref{thm:21BerLiTh} actually yields a stronger conclusion than initially stated: \emph{Under the given assumptions, and further assuming that the interval $I$ is right-unbounded, one obtains the identity}
$$
f(x) ~=~ \sum_{k=0}^{\infty}\frac{f^{(k)}(a)}{k!}\, (x-a)^k\qquad (x\in I)
$$
\emph{for every $a\in I$}. Indeed, in this case, for any $a,x\in I$, it is always possible to choose the point $b\in I$ such that $x<b$ and $|x-a|<b-a$. This representation, in which $x$ is independent of the choice of $a$, provides a stronger statement than the mere real analyticity of $f$ on $I$. Furthermore, since the Taylor series converges on a right-unbounded interval, it must converge on all of $\R$, and thus defines an analytic extension of $f$ to the whole real line. In particular, $f$ is the restriction of a real-analytic function on $\R$, and even of an \emph{entire function} on $\C$. This striking, albeit classical, fact is mentioned, for instance, in Boas \cite[Section 1]{Boa71}.
\end{remark}

Recall that a real-valued function defined on an open interval $I$ is called \emph{regularly monotone} if it is infinitely differentiable and each of its derivatives, including the function itself, preserves a constant sign throughout $I$, regardless of the distribution of these signs.

Bernstein \cite[pp.\ 196--197]{Ber26}, using a somewhat more intricate proof, showed that real analyticity still holds even when the condition of absolute or complete monotonicity is weakened to regular monotonicity. A detailed discussion of this result is provided in Boas \cite[Section 2]{Boa71}, while a shorter proof based on the even-odd decomposition of functions appears in McHugh \cite{McH75}.

\section{Newton Series Representation of Completely Monotone Functions}

We say that a function $f\colon\, I\to\R$, defined on a right-unbounded open interval $I$, \emph{admits a Newton series expansion on $I$ at a point $a\in I$} if the following identity holds:
\begin{equation}\label{eq:NewExp54}
f(x) ~=~ \sum_{k=0}^{\infty}\tchoose{x-a}{k}\,\Delta^k f(a)\qquad (x\in I).
\end{equation}

When $x-a$ is a nonnegative integer, the identity in \eqref{eq:NewExp54} always holds, and the series clearly reduces to a finite sum. However, as discussed in Graham \emph{et al.}\ \cite[p.~191]{GraKnuPat94}, when $x-a$ is not a nonnegative integer, the identity may fail for an arbitrary function $f\colon\, I\to\R$, even if $f$ is real analytic. For instance, when the right-hand series converges for a specific choice of $a$, it may converge to a value different from $f(x)$. To illustrate this, consider the real-analytic function
$$
f(x) ~=~ \sin(\pi x)\qquad (x>0),
$$
and take $a=1$. In this case, we can readily see that $\Delta^n f(1)=0$ for every $n\in\N$. As a result, the Newton series in \eqref{eq:NewExp54} is identically zero, which clearly does not coincide with $f(x)$ for $x>0$. This example demonstrates that admitting a Newton series expansion is a relatively strong and subtle requirement. For a comprehensive treatment of Newton series expansions, including their extension to the complex plane, the reader may consult, for instance, Gel'fond~\cite{Gel71}.

In this section, we present and prove our main result, stated in the following theorem, which can be interpreted as a Newton-series analog of Bernstein's little theorem on a right-unbounded open interval (see Remark~\ref{rem:entireRe}).

\begin{theorem}\label{thm:MainNew35}
Let $I$ be a real right-unbounded open interval. Suppose that a function $f\colon\, I\to\R$ is infinitely differentiable and that $f^{(q)}$ is completely monotone for some $q\in\N$. Then, for any $a\in I$, the function $f$ admits the Newton series expansion:
$$
f(x) ~=~ \sum_{k=0}^{\infty}\tchoose{x-a}{k}\,\Delta^kf(a)\qquad (x\in I)
$$
and the convergence of the series is uniform on compact subsets of $I$.
\end{theorem}

To prove this theorem, we begin by introducing some preliminary definitions and results. The following proposition offers the `Newtonian analog' of Taylor's theorem, with the remainder expressed in terms of a divided difference. While this result is considered folklore in Newton interpolation theory, we include a proof for the sake of clarity and completeness.

For any function $f\colon\, I\to\R$ and any pairwise distinct points $x_0,x_1,\ldots,x_n\in I$, the expression
$$
f[x_0,x_1,\ldots,x_n]
$$
denotes the \emph{divided difference} of $f$ at the points $x_0,x_1,\ldots,x_n$; that is, the coefficient of $x^n$ in the interpolating polynomial of $f$ at the nodes $x_0,x_1,\ldots,x_n$. This concept extends naturally to the case where some of the nodes coincide. Specifically, if a point $x_i$ appears $m$ times among the interpolation nodes, the corresponding divided difference still exists, provided that the derivative $f^{(m-1)}(x_i)$ exists.

\begin{proposition}\label{prop:NewAn56Tay}
Let $f\colon\, I\to\R$ be a function defined on a real right-unbounded open interval. Let also $a,x\in I$, and let $n\in\N$. If $x\in\{a,a+1,\ldots,a+n-1\}$, then
$$
f(x)-\sum_{k=0}^{n-1}\tchoose{x-a}{k}\,\Delta^kf(a) ~=~ 0.
$$
If $x\notin\{a,a+1,\ldots,a+n-1\}$ or if $f$ is differentiable at $x$, then
$$
f(x)-\sum_{k=0}^{n-1}\tchoose{x-a}{k}\,\Delta^kf(a) ~=~ (x-a)^{\underline{n}}\, f[a,a+1,\ldots,a+n-1,x].
$$
\end{proposition}

\begin{proof}
If $x\in\{a,a+1,\ldots,a+n-1\}$, then it is clear that the first identity holds. The second identity also holds in this case provided that $f$ is differentiable at $x$, since both sides then vanish and therefore coincide. Thus, we may now assume that $x\notin\{a,a+1,\ldots,a+n-1\}$.

Let us prove the identity by induction on $n$. It holds trivially for $n=0$ since in that case $f(x)=f[x]$. Assume it holds for some $n\in\N$, and let us show that it remains true for $n+1$. Using the classical recursive formula for divided differences, we obtain
\begin{eqnarray*}
\lefteqn{(x-a)^{\underline{n+1}}\, f[a,a+1,\ldots,a+n,x]}\\
&=& (x-a)^{\underline{n}}\, f[a,a+1,\ldots,a+n-1,x]-(x-a)^{\underline{n}}\, f[a,a+1,\ldots,a+n].
\end{eqnarray*}
Now, applying the induction hypothesis to the first term and using the well-known expression for the second term in terms of finite differences, the right-hand side simplifies to:
$$
f(x)-\sum_{k=0}^{n-1}\tchoose{x-a}{k}\,\Delta^kf(a)-\tchoose{x-a}{n}\,\Delta^nf(a) ~=~  f(x)-\sum_{k=0}^n\tchoose{x-a}{k}\,\Delta^kf(a).
$$
This completes the inductive step and proves the proposition.
\end{proof}

We now recall the definitions of functions that are convex and concave of order~$p$, for any integer $p\geq -1$. These definitions are based on divided differences with $p + 2$ arguments. While some authors use $p + 1$ arguments instead, this notational variation is merely a matter of convention and does not imply a difference in meaning. For background, see, for instance \cite[Section 2.2]{MarZen22}.

\begin{definition}\label{de:B22-pconvconc52}
Let $I$ be any real interval and let $p\geq -1$ be an integer. A function $f\colon\, I\to\R$ is said to be \emph{$p$-convex} (resp.\ \emph{$p$-concave}) if for any system $x_0<x_1<\cdots < x_{p+1}$ of $p+2$ points in $I$ the following inequality holds:
$$
f[x_0,x_1,\ldots,x_{p+1}]~\geq ~0\qquad (\text{resp.}~f[x_0,x_1,\ldots,x_{p+1}]~\leq ~0).
$$
We denote by $\cK^p_1(I)$ (resp.\ $\cK^p_{-1}(I)$) the set of functions $f\colon\, I\to\R$ that are $p$-convex (resp.\ $p$-concave), and we introduce the notation
$$
\cK^p(I) ~=~ \cK^p_1(I)\cup\cK^p_{-1}(I).
$$
\end{definition}

Thus defined, a function $f\colon\, I\to\R$ is $1$-convex if it is convex in the usual sense, $0$-convex if it is increasing, and $({-1})$-convex if it is nonnegative.

A fundamental property worth noting is that if $I$ is an open real interval, then the following inclusion holds:
\begin{equation}\label{eq:Incl32p}
\cK^{p+1}(I) ~\subset ~\cC^p(I)\qquad (p\in\N).
\end{equation}
This result, along with other useful properties stated in the following proposition, can be found in \cite[Section 2.2]{MarZen22} and the references therein.
 
\begin{proposition}\label{prop:Kuczma8}
Let $I$ be an open real interval and let $p\in\N$. Then the following assertions hold.
\begin{enumerate}
\item[(a)] If $f$ lies in $\cK^p_1(I)$, where $I$ is right-unbounded, then $\Delta f$ lies in $\cK^{p-1}_1(I)$.
\item[(b)] If $f\colon\, I\to\R$ is differentiable, then $f$ lies in $\cK^p_1(I)$ if and only if $f'$ lies in $\cK^{p-1}_1(I)$
\item[(c)] If $f$ lies in $\cC^p(I)\cap\cK^p_1(I)$, then the map
$$
(z_0,z_1,\ldots,z_p) ~\mapsto ~ f[z_0,z_1\ldots,z_p]
$$
from $I^{p+1}$ to $\R$ is continuous and increasing in each place.
\end{enumerate}
\end{proposition}

According to Definition~\ref{de:B22-pconvconc52}, it is clear that a function $f\colon\, I\to\R$, defined on an open interval $I$, is completely monotone if and only if
$$
f ~\in ~ \cC^{\infty}(I)\quad\text{and}\quad (-1)^n f^{(n)} ~\in ~\cK^{-1}_{1}(I) \quad\text{for all $n\in\N$}.
$$
By applying Eq.~\eqref{eq:Incl32p} and Proposition \ref{prop:Kuczma8}(b), we see that this latter condition exactly means that
$$
f ~\in ~\cK^{n-1}_{(-1)^n}(I) \quad\text{for all $n\in\N$}.
$$
Similarly, the function $f\colon\, I\to\R$ is absolutely monotone if and only if
$$
f ~\in ~ \cK^{n-1}_1(I) \quad\text{for all $n\in\N$},
$$
and it is regularly monotone if and only if
$$
f ~\in ~ \cK^{n-1}(I) \quad\text{for all $n\in\N$}.
$$
Interestingly, these observations allow us to succinctly reformulate the properties of complete, absolute, and regular monotonicity in terms of higher-order convexity, without requiring the assumption of infinite differentiability.

Before proving Theorem~\ref{thm:MainNew35}, we need to present the following additional preliminary and technical lemma.

\begin{lemma}\label{lemma:546-Levrie}
Let $a,b\in\R$, with $a>b$. Then, the sequence of functions
$$
n ~\mapsto ~ \bigg|\frac{(x-a)^{\underline{n}}}{(b-a)^{\underline{n}}}\bigg|
$$
converges uniformly to zero on compact subsets of $(b,\infty)$.
\end{lemma}

\begin{proof}
Let $a,b,c,d\in\R$, with $a>b$ and $b<c<d$, and let $x\in [c,d]$. Define also $k=\lceil x-a\rceil$, if $x-a>0$, and $k=0$, if $x-a\leq 0$. Thus, if $x-a>0$, the number $k$ is the unique nonnegative integer such that
$$
x-a-k ~\leq ~ 0 ~< ~ x-a-k+1.
$$

For large $n$, we then obtain
$$
\bigg|\frac{(x-a)^{\underline{n}}}{(b-a)^{\underline{n}}}\bigg| ~=~ \frac{(x-a)^{\underline{k}}~\big|(x-a-k)^{\underline{n-k}}\big|}
{\big|(b-a)^{\underline{n}}\big|} ~\leq ~
\frac{(d-a)^{\underline{k}}~\big|(c-a-k)^{\underline{n-k}}\big|}
{\big|(b-a)^{\underline{k}}\big|\,\big|(b-a-k)^{\underline{n-k}}\big|}\, .
$$

Since $0\leq k<\lceil d-b\rceil$, there exists a constant $C>0$ such that
$$
\sup_{x\in [c,d]}~\bigg|\frac{(x-a)^{\underline{n}}}{(b-a)^{\underline{n}}}\bigg| ~\leq ~
C\,\max_{0\,\leq\, k<\lceil d-b\rceil}\,\bigg|\frac{(c-a-k)^{\underline{n-k}}}
{(b-a-k)^{\underline{n-k}}}\bigg|\, .
$$

Now, using the following classical asymptotic equivalence (see, e.g., Levrie \cite{Lev17}):
$$
\tchoose{z}{m} ~\sim ~ \frac{(-1)^m}{m^{z+1}\,\Gamma(-z)}\qquad\text{as $m\to\infty$, $m\in\N^*$},
$$
which holds for any $z\in\C\setminus\N$, we can derive the following one, valid for any integer $k$ such that $0\leq k<\lceil d-b\rceil$:
$$
\frac{(c-a-k)^{\underline{n-k}}}{(b-a-k)^{\underline{n-k}}} ~=~
\frac{\tchoose{c-a-k}{n-k}}{\tchoose{b-a-k}{n-k}} ~\sim ~ \frac{\Gamma(k+a-b)}{\Gamma(k+a-c)}\,\frac{1}{(n-k)^{c-b}}\qquad\text{as $n\to\infty$},
$$
where we may assume that $c-a-k$ is not a nonnegative integer (otherwise, the left-hand expression is identically zero). Since $c-b>0$, the right-hand expression converges to zero as $n\to\infty$. This completes the proof of the lemma.
\end{proof}

We are now prepared to prove Theorem~\ref{thm:MainNew35}. As noted in the previous section, the proof closely mirrors the structure and techniques used in the classical setting.

\begin{proof}[Proof of Theorem~\ref{thm:MainNew35}]
Let $a,x\in I$ and let $n\in\N^*$, with $n\geq q$. Using Proposition~\ref{prop:NewAn56Tay}, we can immediately derive the formula
\begin{equation}\label{eq:NewEq554}
f(x)-\sum_{k=0}^{q-1}\tchoose{x-a}{k}\,\Delta^kf(a) ~=~ \sum_{k=q}^{n-1}\tchoose{x-a}{k}\,\Delta^kf(a) + R_n(x),
\end{equation}
where
$$
R_n(x) ~=~ (x-a)^{\underline{n}}\, f[a,a+1,\ldots,a+n-1,x].
$$

The assumptions on $f$ imply that both functions $(-1)^{n-q}f^{(n)}$ and $(-1)^{n-q+1}f^{(n+1)}$ lie in $\cK_{1}^{-1}(I)$. By Proposition \ref{prop:Kuczma8}(b), this means that the function $f$ lies in the intersection set
$$
\cK^{n-1}_{(-1)^{n-q}}(I) ~\cap ~ \cK^n_{(-1)^{n-q+1}}(I).
$$
Proposition \ref{prop:Kuczma8}(c) then tells us that the map
$$
x~\mapsto ~ f[a,a+1,\ldots,a+n-1,x]
$$
is either negative and continuously increasing or positive and continuously decreasing.

Now, let also $b\in I$, with $b<\min\{a,x\}$. Using the latter observation, we immediately derive the following inequalities:
\begin{eqnarray}
0 ~ \leq ~ |R_n(x)| &=& \big|(x-a)^{\underline{n}}\big|\,\big|f[a,a+1,\ldots,a+n-1,x]\big|\nonumber\\
& \leq & \big|(x-a)^{\underline{n}}\big|\,\big|f[a,a+1,\ldots,a+n-1,b]\big|\nonumber\\
&=& \bigg|\frac{(x-a)^{\underline{n}}}{(b-a)^{\underline{n}}}\bigg|\, |R_n(b)|.\label{eq:CoeffRnb}
\end{eqnarray}

On the other hand, for any integer $k\geq q$, we have
$$
\tchoose{b-a}{k}\,\Delta^kf(a) ~=~ \frac{(b-a)^{\underline{k}}}{k!}\,\Delta^kf(a) ~=~ \frac{\big|(b-a)^{\underline{k}}\big|}{k!}\, (-1)^k\Delta^kf(a)
$$
and this expression has the sign of $(-1)^q$, since the function $(-1)^kf$ lies in $\cK^{k-1}_{(-1)^q}(I)$ (which implies that $(-1)^k\Delta^kf$ lies in $\cK^{-1}_{(-1)^q}(I)$ by Proposition \ref{prop:Kuczma8}(a)). Moreover, we have
\begin{eqnarray*}
R_n(b) &=& (b-a)^{\underline{n}}\, f[a,a+1,\ldots,a+n-1,b]\\
&=& \big|(b-a)^{\underline{n}}\big|\ (-1)^n\, f[a,a+1,\ldots,a+n-1,b],
\end{eqnarray*}
which also has the sign of $(-1)^q$, since the function $(-1)^nf$ lies in $\cK^{n-1}_{(-1)^q}(I)$.

Since the identity in \eqref{eq:NewEq554} also holds when $x=b$, we then obtain
$$
\bigg|f(b)-\sum_{k=0}^{q-1}\tchoose{b-a}{k}\,\Delta^kf(a)\bigg| ~=~ \bigg|\sum_{k=q}^{n-1}\tchoose{b-a}{k}\,\Delta^kf(a)\bigg|+|R_n(b)|
$$
and hence
$$
|R_n(b)| ~\leq ~ \bigg|f(b)-\sum_{k=0}^{q-1}\tchoose{b-a}{k}\,\Delta^kf(a)\bigg|.
$$
By combining this latter inequality with the double inequality in \eqref{eq:CoeffRnb}, we derive the following alternative estimate, where only the first factor depends on $n$:
$$
0 ~\leq ~ |R_n(x)| ~\leq ~ \bigg|\frac{(x-a)^{\underline{n}}}{(b-a)^{\underline{n}}}\bigg|~ \bigg|f(b)-\sum_{k=0}^{q-1}\tchoose{b-a}{k}\,\Delta^kf(a)\bigg|.
$$

Lemma~\ref{lemma:546-Levrie} then implies that the sequence $n\mapsto R_n(x)$ converges uniformly to zero on compact subsets of $(b,\infty)$. Since $b<\min\{a,x\}$ was arbitrary, it follows that the convergence to zero is uniform on compact subsets of the entire interval $I$.
\end{proof}

\begin{remark}
Although not widely known, every Newton series converges uniformly on compact subsets of its domain of convergence (see Gel'fond~\cite[Chapter~2]{Gel71}). For this reason, we have explicitly verified this fact in the proof of Theorem~\ref{thm:MainNew35}.
\end{remark}

\begin{example}\label{ex:1x}
We can readily see that the restriction of the reciprocal function $f(x)=1/x$ to $\R_+$ is completely monotone. By Theorem~\ref{thm:MainNew35}, it follows that this function has a Newton series expansion on $\R_+$ at every point $a>0$. Now, for any $k\in\N$ and any $a>0$, we have (see also Graham \emph{et al.}\ \cite[p.~188]{GraKnuPat94})
$$
\Delta^kf(a) ~=~ (-1)^k\,\frac{k!}{a(a+1)~\cdots ~(a+k)} ~=~ \frac{(-1)^k}{a{\,}{a+k\choose k}}{\,}.
$$
Hence, for any $a>0$, the reciprocal function admits the Newton series expansion:
$$
\frac{1}{x} ~=~ \frac{1}{a}\,\sum_{k=0}^{\infty}(-1)^k\,\frac{\tchoose{x-a}{k}}{{a+k\choose k}}\qquad (x>0).
$$
This identity can be regarded as the discrete analog of the Taylor series expansion of $1/x$ about the point $a$, with the notable distinction that this expansion is valid for all $x>0$.
\end{example}

\begin{remark}\label{rem:BinThm}
The classical \emph{binomial theorem} \cite[p.\ 163]{GraKnuPat94} states that the identity
$$
(c+1)^x ~=~ \sum_{k=0}^{\infty}\tchoose{x}{k}\, c^k\qquad (x\in\R),
$$
or equivalently,
$$
(c+1)^x ~=~ (c+1)^a\, \sum_{k=0}^{\infty}\tchoose{x-a}{k}\, c^k\qquad (a,x\in\R),
$$
holds for $-1<c<1$. When $c>1$, the latter series diverges by the ratio test (unless $x-a$ is a nonnegative integer). This result shows that the exponential function
$$
f(x) ~=~ (c+1)^x\qquad (c>-1,~x\in\R)
$$
admits a Newton series expansion on $\R$ at every $a\in\R$ when $-1<c<1$, but not when $c>1$. Moreover, we can easily see that this function is completely monotone if $-1<c\leq 0$, and absolutely monotone if $c\geq 0$.

For example, the completely monotone function $f_1(x)=e^{-x}$ admits a Newton series expansion on $\R$ at every point $a\in\R$, as does the absolutely monotone function $f_2(x)=(e/2)^x$. In contrast, the absolutely monotone function $f_3(x)=e^x$ does not admit such expansions.

Interestingly, the function $f_2$ demonstrates that a function admitting a Newton series expansion at every point $a\in\R$ need not be completely monotone. The function $f_3$ illustrates that an absolutely monotone function need not admit a Newton series expansion. Together, $f_1$ and $f_3$ illustrate that, although the map $x\mapsto -x$ preserves real analyticity, it does not, in general, preserve the existence of Newton series expansions.
\end{remark}

\section{Applications to Principal Indefinite Sums}

By definition, the derivatives of a completely monotone function alternate in sign. While this may seem like a very special property, it in fact offers a natural framework for a broad class of functions $f\colon\,\R_+\to\R$ known as \emph{principal indefinite sums}, whose definition we will now recall.

We say that a function $f\colon\,\R_+\to\R$ is \emph{eventually $p$-convex} (resp.\ \emph{eventually $p$-concave}) for some integer $p\geq -1$ if it lies in $\cK^p_1(I)$ (resp.\ $\cK^p_{-1}(I)$) for some right-unbounded subinterval $I$ of\/ $\R_+$. We denote by $\cK^p_1$ (resp.\ $\cK^p_{-1}$) the set of functions $f\colon\,\R_+\to\R$ that are eventually $p$-convex (resp.\ eventually $p$-concave), and we define
$$
\cK^p ~=~ \cK^p_1\cup\cK^p_{-1}.
$$
For any $p\in\N$, we also let $\cD^p$ denote the set of functions $g\colon\,\R_+\to\R$ such that the sequence $n\mapsto\Delta^p g(n)$ converges to zero. It is known \cite[Theorem 4.14]{MarZen22} that a function $g\in\cC^p(\R_+)$ lies in $\cD^p\cap\cK_1^p$ (resp. $\cD^p\cap\cK_{-1}^p$) if and only if $g^{(p)}$ eventually increases (resp.\ decreases) to zero.

The last two authors established the following theorem in [15, Theorem 3.6] (see also [16, Theorem 1.6]), which, as demonstrated in Example~\ref{ex:Bohr}, constitutes a broad generalization of the classical \emph{Bohr-Mollerup theorem} \cite{BohMol22}.

\begin{theorem}\label{thm:SigmaEx3}
If $g$ lies in $\cD^p\cap\cK^p$ for some $p\in\N$, then there exists a unique solution $f\in\cK^p$, satisfying $f(1)=0$, to the difference equation $\Delta f=g$ on $\R_+$. Moreover,
\begin{equation}\label{eq:ffnp6}
f(x) ~=~ \lim_{n\to\infty}f_n^p[g](x)\qquad (x>0),
\end{equation}
where
$$
f_n^p[g](x) ~=~ \sum_{k=1}^{n-1}g(k)-\sum_{k=0}^{n-1}g(x+k)
+\sum_{j=1}^p\tchoose{x}{j}\,\Delta^{j-1}g(n)\qquad (x>0),
$$
and $f$ is $p$-convex (resp.\ $p$-concave) on any right-unbounded subinterval of\/ $\R_+$ on which $g$ is $p$-concave (resp.\ $p$-convex). Furthermore, the convergence in \eqref{eq:ffnp6} is uniform on any bounded subset of\/ $\R_+$.
\end{theorem}

Let us denote by $f=\Sigma g$ the function $f\colon\,\R_+\to\R$ defined in Theorem~\ref{thm:SigmaEx3}. With this notation, the map $\Sigma$ carries any function
$$
g ~\in ~ \bigcup_{p\geq 0}\, (\cD^p\cap\cK^p)
$$
into the function $f\colon\,\R_+\to\R$ given by Eq.~\eqref{eq:ffnp6}. This definition is well-posed in the sense that if a function $g$ lies in both $\cD^p\cap\cK^p$ and $\cD^q\cap\cK^q$ for some $p,q\in\N$, with $q\geq p$, then we have
$$
\lim_{n\to\infty}\big(f_n^q[g](x)-f_n^p[g](x)\big) ~=~ \lim_{n\to\infty}\,\sum_{j=p+1}^q\tchoose{x}{j}\,\Delta^{j-1}g(n) ~=~ 0 \qquad (x>0).
$$

We define the \emph{principal indefinite sum} \cite[Definition 5.4]{MarZen22} of a function $g$ in the domain of $\Sigma$ to be the class of functions of the form $c+\Sigma g$, where $c\in\R$. As such, principal indefinite sums form a broad class of functions that arise frequently in mathematical analysis. Numerous examples and applications are examined in detail in \cite[Chapters 10--12]{MarZen22}.

\begin{example}[Bohr-Mollerup's Theorem]\label{ex:Bohr}
Applying Theorem~\ref{thm:SigmaEx3} to the basic function $g(x)=\ln x$ with the parameter $p=1$ reduces to the additive version of \emph{Bohr-Mollerup's theorem} \cite{BohMol22}; see also Krull \cite{Kru48,Kru49} and Webster \cite{Web97b} for earlier work on this topic. It states that the log-gamma function $f(x)=\ln\Gamma(x)$ is the unique solution to the equation
$$
\Delta f(x) ~=~ \ln x\qquad (x>0)
$$
that vanishes at $x=1$ and is eventually convex or eventually concave. It also provides the additive form of Gauss' well-known limit:
$$
\ln\Gamma(x) ~=~ \lim_{n\to\infty}\bigg(\sum_{k=1}^{n-1}\ln k-\sum_{k=0}^{n-1}\ln(x+k)+x\ln n\bigg)\qquad (x>0).
$$
Thus, we can write
$$
\Sigma \ln x ~=~ \ln\Gamma(x)\qquad (x>0).
$$
Moreover, with a slight abuse of terminology, one may say that \emph{the principal indefinite sum of the logarithm function is the log-gamma function}.
\end{example}

We can readily see that the inclusion $\cD^p\subset\cD^{p+1}$ holds for all $p\in\N$. A less trivial inclusion, however, is that $\cK^p\subset\cK^{p-1}$ for all $p\in\N$ (see \cite[Proposition 4.7]{MarZen22}). This latter property naturally motivates the introduction of the intersection
$$
\cK^{\infty} ~=~ \bigcap_{p\geq -1}\cK^p.
$$

Theorem~\ref{thm:SigmaEx3} implies that if a function $g$ lies in $\cD^p\cap\cK^p$ for some $p\in\N$, then $\Sigma g$ exists and lies in $\cD^{p+1}\cap\cK^p$. In particular, if $g\in\cD^p\cap\cK^{\infty}$, then $\Sigma g\in\cD^{p+1}\cap\cK^{\infty}$. It is important to note, however, that the condition $g\in\cK^{\infty}$ does not necessarily imply that $g\in\cK^n(I)$ for all sufficiently large $n$---or equivalently, that $g|_I$ has a regularly monotone derivative---where $I$ is a fixed right-unbounded interval of\/ $\R_+$. As observed in \cite[Example 5.13]{MarZen22}, this situation arises, for example, with the function $g\colon\,\R_+\to\R$ defined by
$$
g(x) ~=~ -\frac{1}{x}\,\ln x\qquad (x>0).
$$ 

The following theorem essentially states that if a function $g\colon\,\R_+\to\R$ has a derivative that eventually tends monotonically to zero, and if a higher order derivative is regularly monotone, then the principal indefinite sum $\Sigma g$ exists and both $g$ and $\Sigma g$ possess completely monotone derivatives. In particular, they are real analytic and admit a Newton series expansion on $\R_+$ at every point $a>0$. We begin with a technical lemma, followed by a preparatory proposition.

\begin{lemma}[{see \cite[Corollary 4.19]{MarZen22}}]\label{lemma:44main}
For any $p\in\N$ and any right-unbounded interval $I$ of\/ $\R_+$, the following inclusion holds: $\cD^p\cap\cK^p_{1}(I)\subseteq\cK^{p-1}_{-1}(I)$.
\end{lemma}

\begin{proposition}\label{prop:44maisnSg}
Let $p,q\in\N$ with $p<q$, and let $g\colon\,\R_+\to\R$ be an infinitely differentiable function. Suppose that $g^{(p)}$ eventually tends monotonically to zero, and that $g^{(q)}$ is regularly monotone. Then, the following assertions hold:
\begin{itemize}
\item[(a)] $g^{(n)}$ eventually tends monotonically to zero for every $n\geq p$.
\item[(b)] $g^{(q-1)}$ or $-g^{(q-1)}$ is completely monotone.
\end{itemize}
Consequently, $g$ is real analytic and admits a Newton series expansion on $\R_+$ at every point $a>0$.
\end{proposition}

\begin{proof}
Assume that $g$ satisfies the stated conditions. Then $g\in\cK^{n-1}(\R_+)$ for all $n\geq q\geq p+1$. Moreover, $g\in\cD^n\cap\cK^n$ for all $n\geq p$, which establishes assertion (a).

Negating $g$ if necessary, we may assume that $g^{(q)}\leq 0$, which precisely means that $g\in\cK^{q-1}_{-1}(\R_+)$. Using Lemma~\ref{lemma:44main}, we conclude that $g\in\cK^{q-2}_{1}(\R_+)$. To establish assertion (b), it suffices to show that $g^{(q-1)}$ is completely monotone; that is,
\begin{equation}\label{eq:gSgfAk0}
g ~\in ~ \cK^{n-2}_{(-1)^{n-q}}(\R_+)\quad\text{for all $n\geq q$}.
\end{equation}
We begin by showing that $g\in\cK^{q}_{1}(\R_+)$. Suppose that $g\in\cK^{q}_{-1}(\R_+)$. Then, by Lemma~\ref{lemma:44main}, it follows that $g\in\cK^{q-1}_{1}(\R_+)$. Since $g$ also lies in $\cK^{q-1}_{-1}(\R_+)$, the $q$th derivative of $g$ must be identically zero. Hence, $g$ is a polynomial of degree at most $q-1$, and in particular, $g\in\cK^{q}_{1}(\R_+)$. This argument can be iterated to establish the condition stated in \eqref{eq:gSgfAk0}.

The final part of the proposition follows directly from assertion (b), together with Theorems~\ref{thm:21BerLiTh} and \ref{thm:MainNew35}. Recall that Theorem~\ref{thm:21BerLiTh} remains valid if the assumption that $f^{(q)}$ is absolutely monotone is replaced by the condition that $f^{(q)}$ is completely monotone.
\end{proof}

\begin{theorem}\label{thm:44maisnSg}
Under the assumptions of Proposition~\ref{prop:44maisnSg}, the function $\Sigma g$ exists and is infinitely differentiable. Moreover, the following assertions hold:
\begin{itemize}
\item[(a)] $(\Sigma g)^{(n)}$ eventually tends monotonically to zero for every $n\geq p+1$.
\item[(b)] Both $g^{(q-1)}$ and $(\Sigma g)^{(q)}$, or their negatives, are completely monotone.
\end{itemize}
Consequently, $\Sigma g$ is real analytic and admits a Newton series expansion on $\R_+$ at every point $a>0$.
\end{theorem}

\begin{proof}
By assertion (a) of Proposition~\ref{prop:44maisnSg}, we obtain $g\in\cD^n\cap\cK^n$ for all $n\geq p$. Then, by Theorem~\ref{thm:SigmaEx3}, it follows that $\Sigma g$ exists and satisfies $\Sigma g\in\cD^n\cap\cK^n$ for all $n\geq p+1$, which proves assertion (a).

By assertion (b) of Proposition~\ref{prop:44maisnSg}, we may assume, for instance, that $g^{(q-1)}$ is completely monotone. It then follows that
$$
g ~\in ~ \cD^{n-1}\cap\,\cK^{n-1}_{(-1)^{n-q+1}}\quad\text{for all $n\geq q\geq p+1$}.
$$
Applying Theorem~\ref{thm:SigmaEx3} once again, we obtain
$$
\Sigma g ~\in ~ \cK^{n-1}_{(-1)^{n-q}}(\R_+)\quad\text{for all $n\geq q$}.
$$
Combining this with Eq.~\eqref{eq:Incl32p}, we conclude that $\Sigma g$ is infinitely differentiable and that $(\Sigma g)^{(q)}$ is completely monotone, which establishes assertion (b). The final part of the statement then follows immediately, as in Proposition~\ref{prop:44maisnSg}.
\end{proof}

\begin{remark}\label{rem:Dpq}
Under the assumptions of Proposition~\ref{prop:44maisnSg}, the condition that $g^{(p)}$ eventually tends monotonically to zero is clearly equivalent to requiring that $g\in\cD^p$. In light of this observation, we derive immediately the following consequence:

\smallskip

\noindent\emph{If an infinitely differentiable function $g\colon\,\R_+\to\R$ lies in $\cD^p$ for some $p\in\N$, and if $g^{(q)}$ is regularly monotone for some $q>p$ , then $g^{(q-1)}$ or $-g^{(q-1)}$ is completely monotone.}
\end{remark}

We also present the following corollary, which states that if the assumptions of Theorem~\ref{thm:44maisnSg} hold for the restriction of $g$ to a right-unbounded open interval of\/ $\R_+$, and if $g$ is real analytic (resp.\ admits a Newton series expansion on $\R_+$ at every $a>0$), then the same property holds for $\Sigma g$.

\begin{corollary}
Let $p,q\in\N$ with $p<q$, let $I$ be a fixed right-unbounded interval of\/ $\R_+$, and let $g\colon\,\R_+\to\R$ be a function such that $g|_I$ is infinitely differentiable. Assume further that the following two conditions hold:
\begin{itemize}
\item[(a)] $g|_I^{(p)}$ eventually tends monotonically to zero.
\item[(b)] $g|_I^{(q)}$ is regularly monotone.
\end{itemize}
Suppose also that $g$ is real analytic (resp.\ admits a Newton series expansion on $\R_+$ at every point $a>0$). Then the function $\Sigma g$ exists and is real analytic (resp.\ admits a Newton series expansion on $\R_+$ at every point $a>0$).
\end{corollary}

\begin{proof}
Let $n\in\N^*\cap I$ and let the function $g\colon\,\R_+\to\R$ satisfy the stated conditions. Then by Theorem~\ref{thm:SigmaEx3}, Proposition~\ref{prop:44maisnSg}, and Theorem~\ref{thm:44maisnSg}, the function $\Sigma g$ exists, and both $g$ and $\Sigma g$ are real analytic and admit a Newton series expansion on the interval $(n,\infty)$ at every point $a>n$. To prove the corollary, it is enough to show that these functions are also real analytic (resp.\ admit a Newton series expansion at every point $a>n-1$) on the larger interval $(n-1,\infty)$.

Suppose first that $g$ is real analytic. Then, real analyticity of $\Sigma g$ is immediate from the identity
\begin{equation}\label{eq:gSgfAk2}
\Sigma g(x) ~=~ \Sigma g(x+1)-g(x)\qquad (x>0),
\end{equation}
which expresses $\Sigma g$ in terms of real-analytic functions on $(n-1,\infty)$.

Suppose now that $g$ admits a Newton series expansion on $(n,\infty)$ at every point $a>n$ and let us show that $\Sigma g$ admits a Newton series expansion on $(n-1,\infty)$ at every point $a>n-1$. Let $a,x\in (n-1,\infty)$. Using \eqref{eq:gSgfAk2} and the assumed Newton expansions of both $g$ and $\Sigma g$ on $(n,\infty)$, we obtain:
\begin{eqnarray*}
\Sigma g(x) &=& \sum_{k=0}^{\infty}\tchoose{x+1-a-1}{k}\,\Delta^k\Sigma g(a+1)-\sum_{k=0}^{\infty}\tchoose{x-a}{k}\,\Delta^kg(a)\\
&=& \Sigma g(a+1)-g(a) + \sum_{k=1}^{\infty}\tchoose{x-a}{k}\,\Delta^{k-1} g(a+1)-\sum_{k=1}^{\infty}\tchoose{x-a}{k}\,\Delta^kg(a).
\end{eqnarray*}
Now, using \eqref{eq:gSgfAk2} once more, along with the trivial identity $g(a+1)=\Delta g(a)+g(a)$, we simplify the expression above to obtain:
$$
\Sigma g(x) ~=~ \Sigma g(a) + \sum_{k=1}^{\infty}\tchoose{x-a}{k}\,\Delta^{k-1} g(a) ~=~ \sum_{k=0}^{\infty}\tchoose{x-a}{k}\,\Delta^k\Sigma g(a).
$$
This completes the proof.
\end{proof}

We conclude this section by presenting two noteworthy examples that demonstrate the application of Theorem~\ref{thm:44maisnSg}.

\begin{example}[Log-gamma Function]
Returning to the functions $g(x)=\ln x$ and $\Sigma g(x)=\ln\Gamma(x)$ defined on $\R_+$ (see Example~\ref{ex:Bohr}), it is straightforward to verify that the assumptions of Theorem~\ref{thm:44maisnSg} are satisfied with $p=1$, $q=2$, and $I=\R_+$. It follows that the function $\ln\Gamma(x)$ is real analytic and admits a Newton series expansion on $\R_+$ at every point $a>0$. Taking $a=1$ for instance, we obtain the following Newton series expansion
$$
\ln\Gamma(x) ~=~ \sum_{k=1}^{\infty}\tchoose{x-1}{k}\, (\Delta^{k-1}_t\ln t)\big|_{t=1}\qquad (x>0),
$$
which was already studied by Hermite in 1900 \cite{Her01}; see also Graham \emph{et al.}~\cite[p.~192]{GraKnuPat94} for further details.
\end{example}

\begin{example}[Stern's Series]
The restriction of the function $g(x)=1/x$ to $\R_+$ clearly satisfies the assumptions of Theorem~\ref{thm:SigmaEx3} with $p=0$, and hence the corresponding principal indefinite sum (up to an additive constant) is given by
$$
\Sigma g(x) ~=~ -\frac{1}{x}+\sum_{k=1}^{\infty}\Big(\frac{1}{k}-\frac{1}{x+k}\Big)\qquad (x>0).
$$
It follows (see \cite[Section 10.2]{MarZen22} and its references) that
$$
\Sigma g(x) ~=~ H_{x-1} ~=~ \psi(x)+\gamma\qquad (x>0),
$$
where $x\mapsto H_x$ denotes the \emph{harmonic number function}, $\psi(x)=\Gamma'(x)/\Gamma(x)$ is the \emph{digamma function}, and $\gamma$ is \emph{Euler's constant}.

The function $g(x)=1/x$ on $\R_+$ also satisfies the assumptions of Theorem~\ref{thm:44maisnSg} with $p=0$, $q=1$, and $I=\R_+$. It follows that $\Sigma g(x)$ is real analytic and admits a Newton series expansion on $\R_+$ at every point $a>0$. Taking $a=1$ for instance, we obtain the following series representation (see also Example~\ref{ex:1x})
$$
\Sigma g(x) ~=~ \sum_{k=1}^{\infty}\tchoose{x-1}{k}\,\Delta_x^{k-1}\frac{1}{x}\,\Big|_{x=1} ~=~ \sum_{k=1}^{\infty}\tchoose{x-1}{k}\,\frac{(-1)^{k-1}}{k}\qquad (x>0),
$$
which is commonly referred to as \emph{Stern's series} (see, e.g., \cite[Theorem~3.8]{Sri07}). 
\end{example}

\section{Concluding Remarks}

We have established the following Newtonian analog of Bernstein's little theorem on a right-unbounded open interval:

\smallskip

\noindent\emph{Any real-valued function defined on a right-unbounded open interval admits a Newton series expansion at every point of that interval, provided that the function itself, or one of its derivatives, is completely monotone}.

\smallskip

We have also applied this result---along with Bernstein's little theorem---to the study of principal indefinite sums.

We summarize the main results and observations of this paper in terms of set identities and inclusions as follows. For any $p\in\N$ and any real open interval $I$, define the following classes of functions:
\begin{eqnarray*}
\cAM_{\pm 1}^p(I) &=& \{f\in\cC^{\infty}(I) : \pm f^{(p+n)}\geq 0,~\text{for all $n\in\N$}\},\\
\cCM_{\pm 1}^p(I) &=& \{f\in\cC^{\infty}(I) : \pm(-1)^n f^{(p+n)}\geq 0,~\text{for all $n\in\N$}\},\\
\cRM^p(I) &=& \{f\in\cC^{\infty}(I) : f^{(p+n)}\geq 0~\text{or}~f^{(p+n)}\leq 0,~\text{for all $n\in\N$}\}.
\end{eqnarray*}
In other words, $\cAM_{\pm 1}^p(I)$ (resp.\ $\cCM_{\pm 1}^p(I)$) consists of functions $f\in\cC^{\infty}(I)$ for which $\pm f^{(p)}$ is absolutely monotone (resp.\ completely monotone). Similarly, $\cRM^p(I)$ is the class of functions $f\in\cC^{\infty}(I)$ for which $f^{(p)}$ is regularly monotone.

The following identities provide characterizations of these function classes in terms of higher-order convexity properties:
\begin{eqnarray*}
\cAM_{\pm 1}^p(I) &=& \bigcap_{n\geq p}\,\cK_{\pm 1}^{n-1}(I),\\
\cCM_{\pm 1}^p(I) &=& \bigcap_{n\geq p}\,\cK_{\pm (-1)^{n-p}}^{n-1}(I),\\
\cRM^p(I) &=& \bigcap_{n\geq p}\,\cK^{n-1}(I).
\end{eqnarray*}
Using this notation, Bernstein's little theorem (see Theorem~\ref{thm:21BerLiTh}), along with its extension to regularly monotone functions, can be expressed through the following inclusions:
\begin{eqnarray*}
\liminf_{n\to\infty}\,\cK_{\pm 1}^{n-1}(I) &=& \bigcup_{p\geq 0}\,\cAM^p_{\pm 1}(I) ~\subset ~ \cC^{\omega}(I),\\
\liminf_{n\to\infty}\,\cK^{n-1}(I) &=& \bigcup_{p\geq 0}\,\cRM^p(I) ~\subset ~ \cC^{\omega}(I),
\end{eqnarray*}
where $\cC^{\omega}(I)$ denotes the space of real-analytic functions on $I$. 

Suppose now that $I$ is a right-unbounded open interval of\/ $\R_+$. Then, using Remark~\ref{rem:Dpq}, we obtain the following inclusion:
$$
\cD^p\cap\cRM^q(I)~\subset ~ \cCM^{q-1}_1(I)~\cup ~ \cCM^{q-1}_{-1}(I)\qquad (p,q\in\N,~p<q).
$$

Let also $\mathcal{N}(I)$ denote the class of functions $f\colon\, I\to\R$ that admit a Newton series expansion on $I$ at every point $a\in I$. With this notation, our main result can be expressed through the following inclusions:
$$
\liminf_{n\to\infty}\,\cK_{\pm (-1)^n}^{n-1}(I) ~=~ \bigcup_{p\geq 0}\cCM_{\pm (-1)^p}^p(I) ~\subset ~ \mathcal{N}(I) ~\subset ~ \cC^{\omega}(I),
$$
where the final inclusion follows from a classical result of Gel'fond~\cite{Gel71}, which implies that any (real) Newton series on $I$ defines a real-analytic function on $I$. Lastly, we have also observed (cf.\ Remark~\ref{rem:BinThm}) the following significant fact:
$$
\cRM^0(I) ~ \not\subset ~ \bigcup_{a\in I}\mathcal{N}_a(I),
$$
where, for any $a\in I$, the notation $\mathcal{N}_a(I)$ stands for the class of functions $f\colon\, I\to\R$ that admit a Newton series expansion on $I$ at $a$.

In light of the latter inclusions, it becomes clear that the class
$$
\mathcal{N}(I) ~=~ \bigcap_{a\in I}\mathcal{N}_a(I)
$$
deserves thorough investigation and, ideally, a natural characterization. Given that the theory of Newton series expansions has remained relatively underexplored in recent decades, developing such a characterization would be both a valuable and nontrivial contribution.

\section*{Declaration of competing interest}

The authors declare that they have no conflict of interest.


\end{document}